\documentclass[11pt]{amsart}
\usepackage[new]{old-arrows}
\usepackage[T1]{fontenc}
\usepackage[english]{babel}
\usepackage{kpfonts}

\usepackage{amsmath,amssymb,amsthm,enumerate,xspace}
\usepackage{comment}
\usepackage{accents}
\usepackage[colorlinks=true,citecolor=blue,linkcolor=blue]{hyperref}
\usepackage[all]{xy}
\usepackage{tikz}
\usepackage{mathrsfs}
\usepackage{cleveref}

\numberwithin{equation}{section}

\newtheorem{thm}{Theorem}[section]

\newtheorem*{prob*}{Problem}

\newtheorem{prop}[thm]{Proposition}
\newtheorem{lem}[thm]{Lemma}
\newtheorem{cor}[thm]{Corollary}

\newtheorem{ithm}{Theorem}

\newtheorem{icor}[ithm]{Corollary}

\newtheorem*{iprob*}{Problem}

\theoremstyle{definition}

\newtheorem{defi}[thm]{Definition}
\newtheorem*{defi*}{Definition}
\newtheorem{rem}[thm]{Remark}

\newtheorem{exam}[thm]{Example}
\newtheorem*{acks*}{Acknowledgements}

\newcommand{\ZZ}{\mathbf{Z}}
\newcommand{\PP}{\mathbf{P}}
\newcommand{\RR}{\mathbf{R}}
\newcommand{\NN}{\mathbf{N}}

\newcommand{\CC}{\mathbf{C}}
\newcommand{\QQ}{\mathbf{Q}}

\newcommand{\SL}{\mathbf{SL}}

\newcommand{\PSL}{\mathbf{PSL}}

\newcommand{\se}{\subseteq}
\newcommand{\inv}{^{-1}}

\newcommand{\ol}{\overline}

\DeclareMathOperator{\Homeo}{Homeo}
\DeclareMathOperator{\Prob}{Prob}
\newcommand{\spa}[1]{\dot{#1}}

\title[Comments on piecewise-projective groups]{Some comments on piecewise-projective groups\\ of the line}

\author[Nicolas Monod]{Nicolas Monod}
\address{\'Ecole Polytechnique Fédérale de Lausanne (EPFL)\\
CH–1015 Lausanne,
Switzerland}
\begin{document}

\begin{abstract}
We consider groups of piecewise-projective homeomorphisms of the line which are known to be non-amenable using notably the Carri{\`e}re--Ghys theorem on ergodic equivalence relations. Replacing that theorem by an explicit fixed-point argument, we can strengthen the conclusion and exhibit uncountably many ``amenability gaps'' between various piecewise-projective groups.
\end{abstract}
\maketitle



\begin{flushright}
\begin{minipage}[t]{0.85\linewidth}\itshape\small
This note is dedicated with admiration to Rostislav Grigorchuk at the occasion of his 70th birthday. Slava taught me about Thompson's group when I was a student --- and many times since then.
\end{minipage}
\end{flushright}

\setcounter{tocdepth}{1}
\tableofcontents

\section{Introduction}
The purpose of this note is to revisit and to strengthen the non-amenability of the group $H(\RR)$ of piecewise-projective homeomorphisms of the real line and of many of its subgroups.

The motivation to revisit the proof given in \cite{Monod_PNAS} is that the method it introduced to establish non-amenability relied on the theory of equivalence relations, specifically on a remarkable theorem of Carrière--Ghys \cite{Carriere-Ghys} adressing a conjecture of Connes and Sullivan. We shall show that the non-amenability can be established from first principles and without the Carrière--Ghys theorem. This possibility was alluded to in Remark~11 of \cite{Monod_PNAS}; this time, we give an elementary group-theoretical proof of non-amenability. Namely, in \Cref{sec:basic: non-amen} we exhibit concrete convex compact spaces on which the groups act without fixed point. This applies to $H(\RR)$ and to subgroups defined over arithmetic rings such at $\QQ$, $\ZZ[\sqrt2]$, $\ZZ[1/p]$ and their uncountably many variants.

The two motivations to strengthen non-amenability are, first, that a number of \emph{amenability-like} properties were established in \cite{Monod_PNAS}, most prominently the fact that $H(\RR)$ does not contain non-abelian free subgroups. Secondly, the best-known piecewise-projective group is Thompson's group $F$, for which amenability remains a notorious open question. In fact, our criterion fails precisely for $F$ and we have repeatedly been asked: does the non-amenability of the various other groups imply or suggest anything for Thompson's group?

Towards this question, we shall prove that there are infinitely many ``layers'' of non-amenability stacked atop one another within the subgroups of $H(\RR)$. To any set $S$ of prime numbers, we associate the subgroup $\Gamma_S$ of $H(\RR)$ obtained by restricting the coefficients of the projective transformations to the ring $\ZZ[1/S]$ of $S$-integers. This defines a poset of $2^{\aleph_0}$ subgroups and $F$ lies at the bottom in the group $\Gamma_\varnothing$ with integral coefficient.

\begin{ithm}\label{thm:many:gaps}
Let $S, S'$ be any sets of prime numbers with $S \subsetneqq S'$.

Then $\Gamma_S$ is not co-amenable in $\Gamma_{S'}$.
\end{ithm}

The definition of co-amenability is recalled below (\Cref{sec:less:amen}); informally, this means that $\Gamma_{S'}$ is ``even less amenable'' than $\Gamma_S$.

In particular, we obtain a mathematically well-defined version of this heuristic answer to the above question: no, the non-amenability of our various subgroups of $H(\RR)$ does not give any hint for Thompson's group. Informally: for $S$ non-empty, $\Gamma_S$ is non-amenable \emph{regardless} of the amenability status of $F$. Had $F$ been co-amenable in $\Gamma_S$, then our non-amenability results would have implied the non-amenability of $F$. Contrariwise, if $F$ is non-amenable, then $\Gamma_S$ is still ``even less amenable'' than $F$.

\medskip

The reader might regret that one can nest only countably many subgroups $\Gamma_S$ into a chain of pairwise not co-amenable subgroups, whereas the poset of all $\Gamma_S$ consists of continuum many subgroups of the countable group $H(\QQ)$.

Not to worry: \Cref{thm:many:gaps} follows from a more general statement that indeed allows us to distinguish \emph{any two} $\Gamma_S$ from the perspective of ``mutual non-amenability''. A concrete way to state this is as follows:

\begin{ithm}\label{thm:more:gaps}
Let $S, S'$ be any sets of prime numbers with $S\neq S'$.

Then there exists a convex compact \textup{(}explicit\textup{)} $H(\QQ)$-space in which one and only one of the two subgroups $\Gamma_S$ and $\Gamma_{S'}$ admits a fixed point.
\end{ithm}

As it turns out, the latter statement is formally stronger than \Cref{thm:many:gaps} even when $S \se S'$. This will be explained in \Cref{sec:less:amen} once the notion of \textbf{relative co-amenability} will have been recalled. This ``amenability gap'' between subgroups can also be reflected in the \textbf{relative Furstenberg boundary} $\partial(G, G_1)$ introduced in \cite{Bearden-Kalantar,Monod_furstenberg}; \Cref{thm:more:gaps} implies:

\begin{icor}\label{cor:rel:bnd}
The relative Furstenberg boundaries $\partial(H(\QQ), \Gamma_S)$, where $S$ ranges over all subsets of prime numbers, are pairwise non-isomorphic compact $H(\QQ)$-spaces.
\end{icor}

The next result goes back to completely general rings and states that the non-amenability of $H(A)$ can be strengthened to an amenability gap with respect to the integral group $\Gamma_\varnothing$. In contrast to all other results stated in this introduction, it will be proved using the relational method introduced in \cite{Monod_PNAS}.

\begin{ithm}\label{thm:co}
If $A<\RR$ is any ring other than $\ZZ$, then $\Gamma_\varnothing$ and a fortiori the Thompson group are not co-amenable in $H(A)$.
\end{ithm}

Finally, we mention that the methods used in this note can easily be adapted to some variants of the above groups. As an example, we show in \Cref{thm:Q} below that the Thompson group is also not co-amenable in a larger group of $C^1$-diffeomorphisms.

\medskip

The considerations of relative co-amenability between various subgroups of a given group $G$ can be formulated using some sort of ``spectrum'' $\spa G$ which is a poset functorially attached to $G$. Loosely speaking, the points of $\spa G$ are defined by subgroups of $G$, where two subgroups are identified whenever they concurrently do or do not admit a fixed point in any given convex compact $G$-space, see \Cref{sec:spa}.

Then \Cref{thm:more:gaps} can be reformulated as stating that the poset of all sets of prime numbers embeds fully faithfully into $\spa G$ for $G=H(\QQ)$. More generally, each $\spa\Gamma_S$ contains a copy of the poset of all subsets of $S$, see \Cref{cor:spa}.

\section{Notation and preliminaries}
\label{sec:not}%
The group $H(\RR)$ consists of all homeomorphisms $h$ of the real line for which the line can be decomposed into finitely many intervals so that on each interval, $h$ is given by $h(x) =g(x)=\dfrac{a x+b}{c x+ d}$ for some $g=\begin{pmatrix}a & b\\ c & d\end{pmatrix}$ in $\SL_2(\RR)$. This $g$ depends on the interval; since we consider homeomorphisms of the line, the singularity $x=-d/c$ cannot lie in the corresponding interval.

The element $g$ is locally uniquely defined in $\PSL_2(\RR)$ but we use matrix representatives in $\SL_2(\RR)$ whenever no confusion is to be feared.

We write the projective line $\PP^1_\RR$ as $\RR \cup \{\infty\}$ using projective coordinates $[x:1]$ for $x\in \RR$ and $\infty = [1:0]$, but we also use the signed symbols $\pm \infty$ to clarify the relative position of points in $\RR$ and how they might converge to $\infty$ in $\PP^1_\RR$.

By restricting the coefficients of the matrix $g$ and the breakpoints in $\RR$, we obtain a wide variety of subgroups of $H(\RR)$:

Given a (unital) ring $A<\RR$ and a set $B\se \RR$ such that $B \cup\{\infty \} \se \PP^1_\RR$ is $\SL_2(A)$-invariant, we denote by $H_B(A)$ the subgroup where matrices are chosen in $\SL_2(A)$ and breakpoints in $B$. A first example is $H_\QQ(\ZZ)$, which is isomorphic to Thompson's group $F$, as observed by Thurston (this fact follows by identifying the Stern--Brocot subdivision tree of $\QQ\cup\{\infty\}$ with the dyadic subdivision tree of $[0,1]$ used in the more common PL definition of $F$ \cite{Imbert}; the proof given in \cite[\S7]{Cannon-Floyd-Parry} is a variant of this argument, with piecewise-projective transformations of $[0,1]$ itself).

The easiest way to construct a piecewise-projective transformation of $\RR$ is to cut $\PP^1_\RR$ in two. Any hyperbolic $g\in \SL_2(A)$ has exactly two fixed points in $\PP^1$ which thus divide $\PP^1_\RR$ into two intervals. Then we can define an element of $H(\RR)$ by the identity on the component containing $\infty$ and by $g$ on the other. For that reason, we worked in \cite{Monod_PNAS} with the set $B$ of all such fixed points, and denoted the resulting group simply by $H(A)$.

It is perhaps even simpler to restrict only the matrix coefficients and work with the group of all piecewise-$\SL_2(A)$ homeomorphisms of the line. In the above notation, this group is $H_\RR(A)$. In the main setting of this article, namely $S$-integers $A=\ZZ[1/S]$, these two conventions coincide anyway:

\begin{lem}\label{lem:break}
Let $S$ be any non-empty set of prime numbers. Then all breakpoints of any element of $H(\RR)$ with matrix coefficients in $\ZZ[1/S]$ are fixed points of hyperbolic elements in $\SL_2(\ZZ[1/S])$.

Thus the group $\Gamma_S$ defined in the introduction coincides with the countable group $H(\ZZ[1/S])$.
\end{lem}

\begin{proof}
Write $A=\ZZ[1/S]$. We first note that since $S$ contains at least some prime $p$, the points $\infty$ and $0$ are fixed points of the hyperbolic element $h=\left(\begin{smallmatrix}p&0\\ 0&1/p\end{smallmatrix}\right)$ of $\SL_2(A)$.

Let now $x\in \RR$ be any breakpoint of an element $g\in \Gamma_S$. Considering the matrix representatives $g_-, g_+\in \SL_2(A)$ describing $g$ on the intervals left and right of $x$, we have $g_- x = g_+ x$ by continuity. Therefore $g_- \inv g_+$ is an element of $\SL_2(A)$ fixing $x$. This element is non-trivial since $x$ is a breakpoint; it is therefore hyperbolic or parabolic. In the first case we are done, and in the second case $x$ is rational because the characteristic polynomial of $g_- \inv g_+$ has only one solution and has its coefficients in $A<\QQ$. Since $\SL_2(\ZZ)$ acts transitively on $\QQ\cup\{\infty\}$, there is $k\in \SL_2(\ZZ)$ with $k.0=x$. Now the conjugate $k h k\inv$ is a hyperbolic element of $\SL_2(A)$ fixing $x$.
\end{proof}

Next, we recall the Frankensteinian cut-and-paste from \cite{Monod_PNAS}. Since we aim for elementary proofs of non-amenability, we give both the basic dynamical explanation and the explicit algebraic computation from \cite[Prop.~9]{Monod_PNAS} (especially since the latter has some ``$\neq$'' instead of ``$=$''  in the journal version).

\begin{prop}\label{prop:cut-paste}
For every $g\in \SL_2(\RR)$ and $x\in \RR$ with $gx \neq \infty$, there is a piecewise-projective homeomorphism $h$ of $\RR$ which coincides with $g$ on a neighbourhood of $x$.

Moreover, one can take $h\in H(A)$ whenever $A<\RR$ is a ring containing the coefficients of $g$.
\end{prop}

\begin{proof}[Dynamical explanation]
The homeomorphism $h$ of $\RR$ is obtained as follows. Keep $g$ on a small interval around $x$ and continue everywhere else with some translation $y \mapsto y +n$. Thus we just need our interval boundaries to be two points left and right of $x$ where $g$ coincides with the translation by $n$. The basic dynamics of projective transformations such as $g$ imply that for \emph{any} large enough $n\in \RR$, there are exactly two such points where this holds for $n$ or possibly $-n$. The only exception is if $g$ was already affine, in which case the proposition is trivial.
\end{proof}

We now give a completely explicit algebraic proof to keep track of the coefficients involved, following \cite{Monod_PNAS}. The above translation by $n$ will now correspond to a column operation on the matrix for $g$ and therefore we take $n\in\ZZ$ to remain in $A$.

\begin{proof}[Algebraic proof]
We can assume $g\infty \neq \infty$. Claim: there is $q_0\in \SL_2(A)$ with $q_0 \infty = g \infty$ and with two fixed points $\xi_-, \xi_+\in\RR\se \PP^1_{\RR}$ such that the open interval of $\RR$ determined by $\xi_-, \xi_+$ contains $g x$ but not $g \infty$.

To deduce the proposition from this claim, consider the piecewise-$\SL_2(A)$ homeomorphism $q$ of $\PP^1_{\RR}$ which is the identity on the above interval and is given by $q_0$ on its complement. This is indeed a homeomorphism since the interval is defined by fixed points of $q_0$. Now $h=q\inv g$ fixes $\infty$ and is the desired element $h\in H(A)$. (Notice that the breakpoints of $h$ are $g\inv\xi_\pm$, which are the fixed points of the hyperbolic matrix $g\inv q_0 g$.)

To prove the claim, represent $g$ as $\begin{pmatrix}a & b\\ c & d \end{pmatrix}$ and define $q_0$ by $\begin{pmatrix}a & b + n a\\ c & d + n c \end{pmatrix}$, where $n\in \ZZ$ is an integer yet to be chosen. Then $q_0 \infty = [a:c]= g \infty$ holds. Moreover, $gx \neq \infty$ forces $c\neq 0$ and we can hence assume $c>0$. Therefore, the trace
\[
\tau = a + d + n c \in A
\]
satisfies $\tau^2>4$ as soon as $|n|$ is large enough. Then we have two real eigenvalues $\lambda_\pm= \frac12 (\tau\pm\sqrt{\tau^2-4})$ and the corresponding eigenvectors $\begin{pmatrix}x_\pm\\ c\end{pmatrix}$, where $x_\pm=\lambda_\pm-d-n c$, represent the points $\xi_\pm=[x_\pm:c]$. Now we compute $\lim_{n\to+\infty} \xi_-=-\infty$ and $\lim_{n\to+\infty} \xi_+=[a:c]=g\infty$, with the latter limit approaching from the left because $x_+ < a$. Conclusion: in case $g x < g \infty$, any sufficiently large $n$ will ensure
\[
\xi_-  \ < \   g x \ < \   \xi_+ \ < \   g \infty,
\]
which yields the claim for that case. In the other case, when $g x > g \infty$, the result is obtained in the exact same way but with $n \to-\infty$.
\end{proof}

\section{Amenability and fixed points}

\subsection{Convex compact spaces}
If our goal is to give a simple and transparent proof of non-amenability, we should choose which one of the many equivalent definitions of amenability to use:

\begin{defi}\label{defi:amen}
A group $G$ is \textbf{amenable} if every convex compact $G$-space $K\neq \varnothing$ admits a fixed point.
\end{defi}

It is implicit in this definition that $G$ acts on $K$ by affine homeomorphisms and that the convex compact structure of $K$ is induced by some ambient locally convex Hausdorff topological vector space containing $K$.

\itshape However, our starting point is a very concrete example of $K$ without fixed point, maybe the simplest and best-known such example:\upshape

\begin{exam}[Measures on the projective line]\label{exam:P1}
Let $K=\Prob(\PP^1_\RR)$ be the space of probability measures on $\PP^1_\RR$, endowed with the weak-* topology. This is the pointwise convergence topology when measures are seen as functionals on the space $C$ of continuous functions on $\PP^1_\RR$. Then $K$ is a convex compact $G$-space for $G=\SL_2(\RR)$ and it admits no $G$-fixed point. In fact, if $g\in G$ is a hyperbolic matrix, then any probability measure fixed by $g$ is supported on the two points $\xi_\pm\in \PP^1_\RR$ fixed by $g$, namely the two points defined by the eigenvectors of $g$.  More precisely, given any $x \neq \xi_\pm$, $g^n x$ converges to the eigenvector with largest eigenvalue as $n\to\infty$ (this is particularly clear after diagonalising $g$). Since this forces the same convergence for any compact interval in $\PP^1_\RR \smallsetminus\{\xi_\pm\}$, it implies the statement for measures.

In conclusion, no point of $K$ can be fixed simultaneously by two hyperbolic matrices without common eigenvector. This classical ``ping-pong'' argument shows much more: suitable powers of two such matrices freely generate a free group.
\end{exam}

\begin{rem}\label{rem:P1:cont}
The definition of amenability is generalised to topological groups by requiring the action in \Cref{defi:amen} to be jointly continuous, i.e. the map $G\times K\to K$ must be continuous. This is therefore a weaker condition than the amenability of the underlying abstract group $G$. The action in \Cref{exam:P1} is jointly continuous for the usual topology on $G$. It therefore witnesses that $G$ is non-amenable also when viewed as a topological group. 
\end{rem}

\begin{exam}[Measures on the $p$-adic projective line]\label{exam:P1:Qp}
The exact same arguments from \Cref{exam:P1} hold over $\QQ_p$ for any prime $p$. This shows that $K=\Prob(\PP^1_{\QQ_p})$ is a convex compact $G$-space without fixed point for $G=\SL_2(\QQ_p)$. Note however that the subgroup $\SL_2(\ZZ_p)$ admits a fixed point in $K$ because it is a compact group in the $p$-adic topology for which the action is continuous.
\end{exam}


Next we introduce our only other source of convex compact spaces: a construction that takes a convex compact space $K_0$ and yields a new space $K$ by considering $K_0$-valued maps. We will only apply this construction to the case where $K_0$ is given by \Cref{exam:P1} or \Cref{exam:P1:Qp}.

\begin{exam}[Spaces of functions]\label{exam:f}
Start with a convex compact $G$-space $K_0$ and define $K$ to be the space of all measurable maps $f\colon \PP^1_\RR\to K_0$, where two maps are identified when they agree outside a null-set. 

To define the compact topology on $K$, we assume for simplicity that $K_0$ is metrisable and realised in the dual $C'$ of some Banach space $C$ with the weak-* topology. This is the case of \Cref{exam:P1} and \Cref{exam:P1:Qp}, namely $C=C(\PP^1)$ is the space of continuous functions on the projective line and $C'$ the space of measures. Now $K$ is endowed with the weak-* topology obtained by viewing it in the dual of $L^1(\PP^1_\RR, C)$. The compactness then follows from Alaoglu's theorem.
%
\end{exam}

(Even beyond the simple needs of the present article, the above assumptions would not be restrictive: metrisability is not needed, and the realisation in some $C'$ can always be obtained by taking $C$ to be the space of continuous affine functions on $K$.)

\subsection{Piecewise action on maps}
\label{sec:PW:action}%
We now endow the convex compact spaces of maps from \Cref{exam:f} with group actions. If $K_0$ is a convex compact $G$-space and moreover $G$ acts on $\PP^1_\RR$, then $G$ acts on $f\colon \PP^1_\RR\to K_0$ by $(g.f)(x) = g  ( f(g\inv x) )$, where $x\in \PP^1_\RR$. Thus $f\in K$ is a $G$-fixed point if and only if $f$ is $G$-equivariant as a map. It is understood here that the $G$-action on $\PP^1_\RR$ is supposed \emph{non-singular}, that is, preserves null-sets. This ensures that the $G$-action is both well-defined and by homeomorphisms. 

If for instance $G=\SL_2(\RR)$ and $K_0=\Prob(\PP^1_\RR)$, then this $G$-action on $K$ admits a fixed point, namely the map sending $x$ to the Dirac mass at $x$.

\medskip

The crucial interest of \Cref{exam:f} is that the action defined above makes sense more generally for \emph{piecewise} groups:

Let $H$ be a group of piecewise-$\SL_2(A)$ transformations of $\PP^1_\RR$, where $A<\RR$ is any subring. At this point it is not even important that $h\in H$ should be a homeomorphism; we only need to know that the interior points of the ``pieces'' cover  $\PP^1_\RR$ up to a null-set, which is notably the case if we have finitely many intervals as pieces. It is then well-defined to consider the projective part $h|_x \in \PSL_2(A)$ of $h$ at $x\in \PP^1_\RR$, except for a null-set in $\PP^1_\RR$ that we shall ignore. Notice that for $h, h'\in H$ we have the chain rule $(h h')|_x = (h|_{h' x})(h'|_x)$.

Given now a convex compact $\PSL_2(A)$-space $K_0$, we define the $H$-action on the space $K$ of measurable maps $f\colon \PP^1_\RR\to K_0$ by
\[
(h.f)(x) = h|_{h\inv x}  ( f(h\inv x) ), \kern5mm x\in \PP^1_\RR. 
\]
This $H$-action on $K$ is perfectly suited to the cut-and-paste method recalled in \Cref{sec:not}. Indeed, noting that the chain rule gives $h|_{h\inv x} = (h\inv|_x)\inv$, we see that $f$ is $H$-fixed if and only if
\[
 f(h x) = h|_{x} (  f(x) ) \kern5mm \text{for all $h\in H$ and a.e. $x\in \PP^1_\RR$.}
\]
The key point is that this equation only involves the \emph{local} behaviour of $h$ at $x$. Therefore, it immediately implies the following.

\begin{prop}\label{prop:Frank:fixed}
Suppose that for every $g\in \PSL_2(A)$ and almost every $x\in \PP^1_\RR$ there is $h\in H$ with $h|_x=g$.

Then $f\in K$ is $H$-fixed if and only if it is $\PSL_2(A)$-fixed.
\end{prop}

\begin{proof}
Suppose that $f$ is $H$-fixed. We want to show, given $g\in\PSL_2(A)$ and $x\in \PP^1_\RR$, that $f(g x) = g f(x)$ holds. The element $h$ of the assumption satisfies $f(h x) = h|_{x} f(x)$, which is exactly what is claimed since $h|_x=g$. The converse is tautological.
\end{proof}

\subsection{The fundamental non-amenability argument}
\label{sec:basic: non-amen}%
We already have all the elements to deduce that many piecewise-projective groups $H(A)$ are non-amenable. We begin with the cases of $A=\ZZ[1/p]$ and more generally of $S$-integral coefficients.

\begin{thm}\label{thm:basic:non-amen}
Let $S$ be any non-empty set of prime numbers and choose $p\in S$.

Then the group $\Gamma_S = H(\ZZ[1/S])$ has no fixed point in the convex compact space $K$ of measurable maps $\PP^1_\RR \to \Prob(\PP^1_{\QQ_p})$.

In particular, this group is non-amenable.
\end{thm}

\begin{proof}
It suffices to consider $S=\{p\}$. The cut-and-paste principle of \Cref{prop:cut-paste} shows that we are in the setting of \Cref{prop:Frank:fixed}. Thus it suffices to show that $\SL_2(\ZZ[1/p])$ has no fixed point in the space $K$ of maps $f\colon \PP^1_\RR \to \Prob(\PP^1_{\QQ_p})$.

We write $\Lambda = \SL_2(\ZZ[1/p])$, $G_\RR=\SL_2(\RR)$, $G_{\QQ_p} = \SL_2(\QQ_p)$ and $G=G_\RR \times G_{\QQ_p}$. By elementary reduction theory (recalled in \Cref{rem:reduction} below), the diagonal image of $\Lambda$ in $G$ is discrete and of finite covolume.

We extend the $\Lambda$-action on $K$ to a $G$-action by
\[
((g_1, g_2).f)(x) = g_2 ( f(g_1\inv x)), \kern5mm g_1\in G_\RR, g_2 \in G_{\QQ_p}.
\]
This is a continuous action without fixed points: a map fixed by $G_\RR$ would be constant and its constant value would be $G_{\QQ_p}$-fixed, but $G_{\QQ_p}$ has no fixed points in $\Prob(\PP^1_{\QQ_p})$, see \Cref{exam:P1:Qp}. However, any point fixed by $\Lambda$ would yield another point fixed by $G$ after integration over the quotient $G/\Lambda$. Explicitly, if $f\in K$ is $\Lambda$-fixed, then the orbit map $g\mapsto g.f$ descends to $G/\Lambda$ and hence $\int_{G/\Lambda} g.f\, d(g\Lambda)\in K$ is $G$-fixed.
\end{proof}

\begin{rem}[Reduction theory]\label{rem:reduction}
Since the proof of \Cref{thm:basic:non-amen} used that $\Lambda$ is discrete and of finite covolume in $G_\RR\times G_{\QQ_p}$, it is worth recalling that this is elementary:

Discreteness holds because $\ZZ[1/p]$ is discrete in $\RR \times \QQ_p$ by definition of the $p$-adic topology. As for a fundamental domain, we can take $D\times \SL_2(\ZZ_p)$ whenever $D$ is a fundamental domain for the modular group $\SL_2(\ZZ)$ in $\SL_2(\RR)$. Indeed, $\SL_2(\ZZ_p)$ is a compact-open subgroup of $G_{\QQ_p}$ (again because the corresponding fact holds for $\ZZ_p$ in $\QQ_p$ by definition of the topology) and ${\Lambda \cap \SL_2(\ZZ_p)}$ is $\SL_2(\ZZ)$. Finally, a very explicit domain $D$ of finite volume is familiar since Gauss, namely the pre-image in $\SL_2(\RR)$ of the strip $|{\mathrm{Re}(z)| \leq 1/2}$, ${|z|\geq 1}$ in the upper half-plane.

The case of other number fields, as in \Cref{exam:places} below, is handled by restriction of scalars, see e.g. Siegel, Satz~12 p.~233 in \cite{Siegel39}. (The generalisation to arbitrary reductive groups, which soars far beyond our needs, is the Borel--Harish-Chandra theorem \cite[Thm.~12.3]{Borel-Harish-Chandra}, \cite[\S8]{Borel63}.)
\end{rem}

The proof of \Cref{thm:basic:non-amen} is based on the $p$-adic completion of the number field $\QQ$. It holds exactly the same way for the two other types of places, namely real and complex completions:

\begin{exam}[Other places]\label{exam:places}
The argument given above for the ring $A=\ZZ[1/p]$ can be applied to any other ring $A$ of $S$-integers (or integers) in any other real number field $F<\RR$, except $A=\ZZ$. Indeed, when we used $\QQ_p$ in the proof of \Cref{thm:basic:non-amen}, we only needed some completion of $F=\QQ$ \emph{in addition} to the given completion $\RR$ used to define the action on the variable $x\in\PP^1_\RR$.

\smallskip
In particular, this also works if the second completion happens to be $\RR$ as well. For instance, consider $A=\ZZ[\sqrt 2]$, which is the ring of integers of $\QQ(\sqrt2)$ \cite[\S2.5]{Samuel_ANT}. We let $\Lambda=\SL_2(A)$ act on $\Prob(\PP^1_{\RR})$ via the \emph{second} embedding $\Lambda \to G_\RR$ given by the negative root of two. Reasoning exactly as above, the action of $\Lambda$ on the space $K$ of maps $\PP^1_\RR \to \Prob(\PP^1_{\RR})$ has no fixed points because $\Lambda$ is a lattice in $G_\RR \times G_\RR$. Therefore, \Cref{prop:Frank:fixed} shows that $H(A)$ has no fixed points in $K$ and in particular it is non-amenable. 

\smallskip
Likewise, if a real number field $F<\RR$ is not totally real, it can happen that the only other Archimedean completion is complex. This is the case for instance of pure cubic fields, e.g.\ $F=\QQ(\sqrt[3]2)$. Denoting its ring of integers by $\mathscr{O}_F$, we can thus use that $H(\mathscr{O}_F)$ has no fixed points in the convex compact space of maps $\PP^1_\RR \to \Prob(\PP^1_{\CC})$.
\end{exam}

\begin{rem}[Produced spaces]
It is well-known that if a (discrete) group $\Gamma$ contains a non-amenable group $\Lambda<\Gamma$, then $\Gamma$ is also non-amenable. With $\Gamma=H(\RR)$ in mind, the reader might ask: does this fact also admit an elementary proof in terms of fixed points in convex compact spaces? The question is legitimate since this fact has no analogue for general topological groups, whereas the fixed-point condition is completely general.

There is indeed such a direct argument. Let $K_0\neq\varnothing$ be a convex compact $\Lambda$-space without fixed points. Define the \textbf{produced space} $K$ by
\[
K = \big\{ f\colon \Gamma \to K_0 : f(h g) = h f(g) \ \forall\, h\in \Lambda, g\in\Gamma \big\}
\]
which is convex and compact in the product topology, i.e.\ the topology of pointwise convergence for maps $\Gamma \to K_0$. This is a $\Gamma$-space for the right translation action $(\gamma f)(g) = f (g\gamma)$. Now any $\Gamma$-fixed point $f$ in $K$ would be constant, and by construction the constant value $f(g)$ would be a $\Lambda$-fixed point in $K_0$.

Our only scruple could be to verify that $K$ is non-empty; we note that any transversal $S\se \Gamma$ for $\Lambda$ gives by restriction an isomorphism $K \cong K_0^S$.
\end{rem}

\section{Refining and strengthening non-amenability}
\subsection{Relativisation}
\label{sec:less:amen}%
A key concept here is the notion of \textbf{co-amenability} for a subgroup $G_1<G$ of a group $G$, due to Eymard \cite{Eymard72}. Informally, it means that $G$ is ``as amenable as $G_1$'', or ``amenable conditioned on $G_1$ being so''. Neither $G_1$ nor $G$ need be amenable when the question of co-amenability is raised. If for instance $G_1$ is a normal subgroup of $G$, then co-amenability amounts simply to the amenability of the quotient group $G/G_1$.

To motivate the general definition, recall that a group $G$ is amenable iff every non-empty convex compact $G$-set admits a $G$-fixed point.

\begin{defi}\label{defi:co-amen}
A subgroup $G_1<G$ of a group $G$ is \textbf{co-amenable} in $G$ if every convex compact $G$-set with a $G_1$-fixed point admits a $G$-fixed point.
\end{defi}

This notion, which makes sense in the generality of topological groups and their subgroups, has been extensively studied by Eymard \cite{Eymard72} in the context of locally compact groups.

\medskip
There is a further generalisation (\cite[\S2.3]{Portmann_PhD}, \cite[\S7.C]{Caprace-Monod_rel}) that compares two subgroups $G_1, G_2<G$ that are not necessarily nested:

\begin{defi}
The subgroup $G_1$ is \textbf{co-amenable to $G_2$ relative to $G$} if every convex compact $G$-space which has a $G_1$-fixed point also has a $G_2$-fixed point.
\end{defi}

Again, this definition extends to arbitrary topological groups simply by requiring that all actions on convex compact spaces be jointly continuous.

Back to the discrete case, we record the following standard equivalences.

\begin{lem}\label{lem:relco}
Given two subgroups $G_1, G_2$ of a group $G$, the following are equivalent:
\begin{enumerate}[(i)]
\item $G_1$ is co-amenable to $G_2$ relative to $G$;\label{pt: relco}
\item the $G_2$-action on $G/G_1$ admits an invariant mean;\label{pt: relco:mean}
\item the restriction to $G_2$ of the quasi-regular $G$-representation $\lambda_{G/G_1}$ weakly contains the trivial representation.\label{pt: relco:rep}
\end{enumerate}
\end{lem}

The reformulation \eqref{pt: relco:rep} suggests to use \Cref{thm:more:gaps} as an input for C*-algebra arguments; this is pursued in \cite{Gerasimova-Monod_ex}.

\begin{proof}[Proof of \Cref{lem:relco}]
The equivalence of \eqref{pt: relco:mean} and \eqref{pt: relco:rep} is the usual Hulanicki--Reiter condition for any action on any set, here the $G_2$-action on $G/G_1$. The implication \eqref{pt: relco}$\Rightarrow$\eqref{pt: relco:mean} follows by considering the convex compact $G$-space of all means on $G/G_1$, noting that it indeed has a $G_1$-fixed point, namely the Dirac mass at the trivial coset.

Finally, for \eqref{pt: relco:mean}$\Rightarrow$\eqref{pt: relco}, consider an arbitrary convex compact $G$-space $K$ which contains some $G_1$-fixed point $x\in K$. The orbital map $g\mapsto g x$ induces a $G$-equivariant map $G/G_1\to K$ and hence by push-forward we obtain a $G_2$-invariant mean on $K$. The barycentre of this mean is a $G_2$-fixed point. (We recall here that the usual notion of barycentre for probability measures applies to means. Indeed, probability measures are states on continuous functions, while means are states on bounded functions; but any continuous function on $K$ is bounded.)
\end{proof}

With the statements of \Cref{thm:many:gaps} and \Cref{thm:more:gaps} in mind, the importance of relative co-amenability is as follows. By exhibiting a large poset of subgroups $G_i<G$ that are pairwise \emph{not} co-amenable to one another relative to $G$, we strengthen the non-amenability of $G$. Every chain of inclusions in such a poset can be thought of as a gradient of increasing non-amenability.

We underline that even when the subgroups are nested $G_1< G_2<G$, this relative non-amenability is stronger than simply stating that $G_1$ is not co-amenable in $G_2$:

\begin{exam}\label{exam:MPP}
Consider a situation with $G_1< G_2<G$ where $G_1$ is co-amenable in $G$ but not in $G_2$. Examples are given in \cite{Monod-Popa} and \cite{Pestov}, for instance $G=(\bigoplus_\ZZ F_2)\rtimes \ZZ$, $G_2=\bigoplus_\ZZ F_2$ and $G_1=\bigoplus_\NN F_2$. Then $G_1$ is still co-amenable to $G_2$ \emph{relative to $G$}.
\end{exam}

We can now complete the proof of both \Cref{thm:many:gaps} and \Cref{thm:more:gaps} from the introduction. Recall that for any set $S$ of prime numbers, $\Gamma_S$ denotes the group of piecewise-$\SL_2(\ZZ[1/S])$ homeomorphisms of the line and that when $S$ is non-empty, $\Gamma_S$ coincides with $H(\ZZ[1/S])$. The more explicit statement is as follows:

\begin{thm}\label{thm:all:gaps}
Let $S,S'$ be any sets of primes. If $p$ is in $S' \smallsetminus S$, then the convex compact $H(\QQ)$-space of measurable maps $\PP^1_\RR \to \Prob(\PP^1_{\QQ_p})$ admits a $\Gamma_{S}$-fixed point but no $\Gamma_{S'}$-fixed point.

In particular, $\Gamma_{S}$ is not co-amenable to $\Gamma_{S'}$ relative to $H(\QQ)$.

A fortiori, if $S \subsetneqq S'$ then the subgroup $\Gamma_S<\Gamma_{S'}$ is not co-amenable in $\Gamma_{S'}$. 
\end{thm}

\begin{proof}
We proved in \Cref{thm:basic:non-amen} that $\Gamma_{S'}$ has no fixed point. By contrast, $\Gamma_{S}$ fixes the constant map whose value is the $\SL_2(\ZZ_p)$-invariant measure on $\PP^1_{\QQ_p}$ (see \Cref{exam:P1:Qp}) since $\ZZ[1/S]$ is contained in $\ZZ_p$.
\end{proof}

The fact that \Cref{cor:rel:bnd} follows from \Cref{thm:all:gaps} is established in Corollary~4.2 of \cite{Monod_furstenberg}.

\medskip

As discussed in the introduction, the above results also illustrate that the open problem of the (non-)amenability of Thompson's group $H_{\QQ}(\ZZ)$ is completely decoupled from the non-amenability of the various groups $H(A)$ with $A$ non-discrete. In the above poset, $H_{\QQ}(\ZZ)$ lies in $\Gamma_{\varnothing}$ and is therefore not co-amenable in any of the other $\Gamma_S$, nor even co-amenable to $\Gamma_S$ relative to $H(\QQ)$.

\subsection{Other rings}
The goal of this note until this point was to establish variations and stronger forms of the non-amenability theorem from \cite{Monod_PNAS}, but without the theory of measured equivalence relations. Instead, we argued from first principles using simple cut-and-paste and leveraging the very basic fact that $\SL_2(\ZZ)$ has finite covolume in $\SL_2(\RR)$.

This was enough to be able to deal with the rings $A=\ZZ[1/S]$ and applies also to the ring of integers in any other real number fields, such as $A=\ZZ[\sqrt2]$, or their $S$-arithmetic generalisations.

\medskip
In order to justify \Cref{thm:co} from the introduction even when $A$ is not some $S$-arithmetic ring, we prove the following, which does rely on the Carrière--Ghys theorem.

\begin{thm}\label{thm:coamen:CG}
Let $\Lambda<H(\RR)$ be any subgroup containing $\Gamma_\varnothing$. Suppose that $\Lambda$ has the same same orbits in $\PP^1_\RR$ as a countable dense subgroup $\Delta<\SL_2(\RR)$ (after discarding a null-set in $\PP^1_\RR$).

Then $\Gamma_\varnothing$ is not co-amenable in $\Lambda$.

The same statement holds with $\Gamma_\varnothing$ replaced throughout by Thompson's group.
\end{thm}

This can be applied for instance when $\Lambda$ is the group studied by Lodha--Moore \cite{Lodha-Moore}.

\medskip

To prove \Cref{thm:coamen:CG}, we shall need the following general fact about relative amenability in relation to relations; this is a relative of the fact related in Proposition~9 of \cite{Monod_wreath}.

\begin{prop}\label{prop:rel:rel}
Let $\Lambda$ be a countable group with a non-singular action on a standard probability space $X$. Let $\Gamma < \Lambda$ be a co-amenable subgroup of $\Lambda$.

If the orbit equivalence relation produced by $\Gamma$ on $X$ is amenable, then so is the relation produced by $\Lambda$.
\end{prop}

\begin{proof}[Proof of \Cref{prop:rel:rel}]
Let $L\neq\varnothing$ be a compact metrisable space and let $\alpha\colon \Lambda\times X \to \Homeo(L)$ be a \textbf{relation cocycle} to the group of homeomorphisms of $L$. Being a \emph{cocycle} means that  $\alpha(\lambda, \cdot)$ is a measurable map for each $\lambda\in \Lambda$ and that for $\lambda,\eta\in \Lambda$ the chain rule $\alpha(\lambda \eta, x)= \alpha(\lambda, \eta x) \alpha(\eta, x)$ holds for a.e.~$x\in X$. Being a \emph{relation} cocycle means that $\alpha(\lambda, x)$ depends only on the pair $(\lambda x, x)$ of points.

Hjorth proved (Theorem~1.1 in \cite{Hjorth_furstenberg}) that the relation produced by $\Lambda$ on $X$ is amenable if and only if for any such $L$ and $\alpha$, there exists a measurable map $f \colon X \to \Prob(L)$ such that for every $\lambda\in\Lambda$ and a.e.~$x\in X$
\[
f (\lambda x) = \alpha(\lambda, x)  f(x).
\]
This is equivalent to $f$ being $\Lambda$-fixed for the action
\[
(\lambda.f)(x) = \alpha(\lambda, \lambda\inv x) ( f (\lambda\inv x)).
\]
Exactly as in Section~\ref{sec:PW:action}, we have now a convex compact $\Lambda$-space $K$ of maps. Hjorth's criterion tells us that $\Gamma$ has a fixed point in $K$, and therefore the co-amenability of $\Gamma$ in $\Lambda$ yields the conclusion.
\end{proof}

\begin{proof}[Proof of \Cref{thm:coamen:CG}]
The Carrière--Ghys theorem states that the orbit equivalence relation produced by $\Delta$ on $\SL_2(\RR)$ is not an amenable equivalence relation. As recalled in \cite{Monod_PNAS}, this implies that the relation produced on $\PP^1_\RR$ is not amenable either. (This follows from the fact that $\PP^1_\RR$ can be realised as the quotient of $\SL_2(\RR)$ by an \emph{amenable} subgroup, namely the upper triangular matrices.) Thus the relation produced by $\Lambda$ is not amenable.

On the other hand, the relation produced by $\Gamma_\varnothing$ is amenable since it coincides with the relation produced by $\SL_2(\ZZ)$, which is amenable because $\SL_2(\ZZ)$ is discrete in $\SL_2(\RR)$, see \cite[4.3.2]{Zimmer84}.

At this point, the statement follows from \Cref{prop:rel:rel}.
\end{proof}

Note that \Cref{thm:coamen:CG} implies in particular \Cref{thm:co} for any countable ring $\ZZ \lneqq A<\RR$.

There remains one issue, namely that the ring is not supposed countable in \Cref{thm:co}. The remedy should be to consider $\Gamma_\varnothing < H(A') < H(A)$ for some countable ring $\ZZ \lneqq A' <A$. The problem with this approach is that given nested subgroups $\Gamma < \Lambda_1 < G$, there is generally no implication between the co-amenability of $\Gamma$ in $\Lambda_1$ or in $G$. Indeed, examples to this end can be found in \cite{Pestov} and \cite{Monod-Popa}, as already mentioned in \Cref{exam:MPP}.

We can overcome this difficulty if we observe that (co-)amenability, being an analytic property, is countably determined:

\begin{prop}\label{prop:countable}
Let $\Gamma< G$ be a countable co-amenable subgroup of a group $G$. 

For every countable $\Lambda_1 < G$ containing $\Gamma$, there is  a countable subgroup $\Lambda<G$ containing $\Lambda_1$ such that $\Gamma$ is co-amenable in $\Lambda$. 


If we are moreover given any cofinal family $\mathscr H$ of countable subgroups $H<G$ which is closed under countable increasing unions, then we can choose $\Lambda \in \mathscr H$.
\end{prop}

\noindent
(\emph{Cofinal} means that every countable subgroup of $G$ is contained in some $H\in\mathscr H$. In our case, $\mathscr H$ will consist of the groups of the form $H(A)$.)

\smallskip
We note that \Cref{prop:countable} has a trivial converse:

Let $\mathscr H$ be a family as above and let $\Gamma<G$ be a countable subgroup. Suppose that for every countable $\Lambda_1 < G$ containing $\Gamma$ there is $\Lambda \in \mathscr H$ containing $\Lambda_1$ such that $\Gamma$ is co-amenable in $\Lambda$. Then $\Gamma$ is co-amenable in $G$.

This directly follows from the definition of co-amenability (\Cref{defi:co-amen}) by a compactness argument.

\begin{proof}[Proof of \Cref{prop:countable}]
We inductively construct an increasing sequence of countable subgroups $\Lambda_n < G$, each enumerated as $\{ \lambda_{n,j} : j \in \NN\}$, and a sequence of functions $v_n \in \ell^1(G/\Gamma)$ such that
\[
\big\| \lambda_{i,j} v_n - v_n \big\| <  \frac1n  \| v_n \| \kern5mm \forall \, i,j \leq n
\]
holds for the quasi-regular (translation) $G$-representation on $\ell^1(G/\Gamma)$. Supposing $\Lambda_n$ already given, the existence of $v_n$ follows from the co-amenability of $\Gamma$ in $G$ (see e.g. \cite[p.~44]{Eymard72}). We then define $\Lambda_{n+1}$ to be the subgroup generated by the union of $\Lambda_n$ and the pre-image in $G$ of the support of $v_n$, noting that the latter is a countable set. If some $\mathscr H$ is given, we replace $\Lambda_{n+1}$ by an element of $\mathscr H$ containing it.

Define now $\Lambda$ to be the union of all $\Lambda_n$. Since the support of every $v_n$ is contained in $\Lambda/\Gamma \se G/\Gamma$, any accumulation point of the sequence $v_n$ in the space of means on  $\ell^\infty(G/\Gamma)$ defines in fact a mean on $\ell^\infty(\Lambda/\Gamma)$. This mean is $\Lambda$-invariant, thus verifying another equivalent characterisation \cite{Eymard72} of the co-amenability of $\Gamma$ in $\Lambda$.
\end{proof}

\begin{proof}[Proof of \Cref{thm:co}]
Suppose for a contradiction that $\Gamma_\varnothing$ is co-amenable in $G=H(A)$ for a ring $\ZZ \lneqq A <\RR$. Let $\mathscr H$ be the family of all $H(A')$, where $A'$ ranges over all countable rings with $\ZZ \lneqq A' <A$. Let $\Lambda_1=H(A')$ be one such group and let $\Lambda=H(A'')$ be the countable group given by \Cref{prop:countable}. Thus $A''$ is a countable ring in $A$ and $\SL_2(A'')$ is dense in $\SL_2(\RR)$ because $A'< A''$. Therefore, the co-amenability of $\Gamma_\varnothing$ in $\Lambda$ contradicts \Cref{thm:coamen:CG}.
\end{proof}

\section{Additional comments}
\subsection{Breaking up smoothly and rationally}
The breakpoint convention chosen in \cite{Monod_PNAS} to define $H(A)$ has an aesthetic drawback in the case of $A=\ZZ$. Namely, the fixed points of hyperbolic elements are surds (solutions of quadratic equations) while the fixed points of parabolic elements are rational. In particular, the analogue of \Cref{lem:break} does not hold in that case.

Thus the Thompson group $H_\QQ(\ZZ)$ and the group $H(\ZZ)$ are two different subgroups of $\Gamma_\varnothing$ (though $H(\ZZ)$ contains isomorphic copies of Thompson's group, see \cite{Stankov_H}).

\medskip
In addition to having rational breakpoints, the Thompson group $H_\QQ(\ZZ)$ exhibits another pleasant quality: since $\ZZ$ has no non-trivial positive unit, these piecewise-projective elements are actually $C^1$-smooth. This is the maximal smoothness for breakpoints because projective transformations are entirely determined by their $C^2$-jet at any point. We could therefore also ask to strengthen our non-amenability results for groups of $C^1$-diffe\-o\-morphisms.

The method of proof employed above can indeed be adapted to this setting; here is an example.

\begin{thm}\label{thm:Q}
The Thompson group $H_\QQ(\ZZ)$ is not co-amenable in the group $H_\QQ(\QQ)$ of piecewise-$\SL_2(\QQ)$ homeomorphisms of the line with rational breakpoints.

It is also not co-amenable in the smaller subgroup $H_\QQ^{C^1}(\QQ)$ of $C^1$-diffeomorphisms.
\end{thm}

All that is needed is to revisit the cut-and-paste of \Cref{prop:cut-paste} and perform a more cosmetic type of surgery:

\begin{prop}\label{prop:smooth:cut}
For every $g\in \SL_2(\QQ)$ and $x\in \RR$ with $gx \neq \infty$, there is a piecewise-$\SL_2(\QQ)$ homeomorphism $h$ of $\RR$ with breakpoints in $\QQ$ and which coincides with $g$ on a neighbourhood of $x$.

Furthermore, we can choose $h$ to be a $C^1$-diffeomorphism of $\RR$.
\end{prop}

\begin{proof}[Proof of \Cref{prop:smooth:cut}]
We first justify why we can take breakpoints in $\QQ$. Thus, in the proof of \Cref{prop:cut-paste}, we want the fixed points $\xi_\pm$ of $q_0$ to be rational. Equivalently, the root $\sqrt{\tau^2-4}$ must be rational, recalling that $\tau$ is the trace $a + d + n c$.

The point is that we were completely free to choose $n$ as long as we can take $|n|$ arbitrarily large and with a sign prescribed by the relative position of $g x$ and $g \infty$. Since we work now in $\SL_2(\QQ)$, we just need to show that there are such $n\in \QQ$ with moreover $\sqrt{\tau^2-4}\in \QQ$.

The solution to this Diophantine problem was already know to Euclid (see Book X, Proposition~29 in the \emph{Elements}), as follows. Given any integer $N$, let define $n\in \QQ$ by $n=(N + 1/N - a-d)/c$. Then $\tau = N+1/N$ and thus $\tau^2-4=(N-1/N)^2$ is indeed a square. Moreover, $n$ can indeed be chosen arbitrarily large of either sign simply by letting $N\to\pm \infty$ in $\ZZ$.

\medskip
We now turn to the $C^1$ condition, which will require additional dissections to assemble a polytomous spline.

The only singularities introduced in the proof of \Cref{prop:cut-paste} arise from the two points $\xi_\pm$, where $q$ has breakpoints, transitioning from $q_0$ to the identity and back. The strategy is to smoothen $q$ near one breakpoint at the time, which is sufficient provided the modification can be done in a small enough neighbourhood of the breakpoint. Since $\xi_\pm$ are rational and $\SL_2(\QQ)$ acts doubly transitively on the rational points of the projective line, it suffices to prove the following claim:

For any $\epsilon>0$ and for any $p_0\in \SL_2(\QQ)$ fixing $0$ and $\infty$, there exists a $C^1$-smooth piecewise-$\SL_2(\QQ)$ homeomorphism $p_1$ of $\RR$ with breakpoints in $\QQ$ and which coincides with the identity on $(-\infty, -\epsilon]$ and with $p_0$ on $[\epsilon, +\infty)$.

The assumptions on $p_0$ imply that it is given by a diagonal matrix, or equivalently that there is $a\in \QQ$ with $p_0(x) = a^2 x$ for all $x$; without loss of generality, $a>0$.  Let $\epsilon_1>0$ be rational with $\epsilon_1 \leq \epsilon, \epsilon/a$ and define $u\in \SL_2(\QQ)$ by
\[
u=\frac{1}{a+1}
\begin{pmatrix}
2 a & \epsilon_1 a (a-1) \\
(1-a)/(\epsilon_1 a) & 2
\end{pmatrix}
\]
(The conceptual explanation for this choice is that $u$ is a unipotent, which allows us to match derivatives.) We now define $p_1$ as follows for $x\in \RR$:
\[
p_1(x) =
\begin{cases}
x & \text{ if } x \in (-\infty, -\epsilon_1 a],\\
u(x) = \dfrac{2 a x + \epsilon_1 a (a-1)}{(1-a) x /(\epsilon_1 a ) + 2} & \text{ if } x \in (-\epsilon_1 a, \epsilon_1],\\
p_0(x) =a^2 x & \text{ if }  x \in (\epsilon_1, +\infty).
\end{cases}
\]
Thus we only have to check that $p_1$ is indeed a $C^1$-smooth homeomorphism. Note first that the denominator in $u(x)$ vanishes only at $x={2 \epsilon_1 a /{(a-1)}}$, which is outside the interval $(-\epsilon_1 a, \epsilon_1]$ where $u$ is applied. Now we turn to the breakpoints, where $p_1$ is continuous since $u( -\epsilon_1 a) =  -\epsilon_1 a$ and $u(\epsilon_1)=a^2 \epsilon_1$. Furthermore, computing
\[
u'(x) = \left(\frac{(1+a) a \epsilon_1}{(1-a) x + 2 a \epsilon_1}\right)^2
\]
we find that $u'(x)=1$ at $x= -\epsilon_1 a$ and $u'(x)=a^2$ at $x= \epsilon_1$. This verifies that $p_1$ is a $C^1$-smooth homeomorphism.
\end{proof}

\begin{rem}
If we consider $h$ as a transformation of  $\PP^1_\RR$ fixing $\infty$ rather than as a transformation of $\RR$, then it remains true that $h$ is $C^1$. Indeed, recall from the proof of \Cref{prop:cut-paste} that $h$ is just a translation in some neighbourhood of $\infty$; this fact has not been altered by the smoothing operation of the above proof. Thus $h$ is projective (in particular smooth) in a neighbourhood of $\infty$ in $\PP^1_\RR$.
\end{rem}

\begin{proof}[Proof of \Cref{thm:Q}]
It suffices to exhibit a convex compact $H_\QQ(\QQ)$-space admitting a $H_\ZZ(\QQ)$-fixed point but no point fixed by the smooth group $H_\QQ^{C^1}(\QQ)$. Choose any prime $p$. Then the space of measurable maps $\PP^1_\RR \to \Prob(\PP^1_{\QQ_p})$ will do. Indeed, \Cref{prop:smooth:cut} allows us to apply \Cref{prop:Frank:fixed}. Thus it suffices to show that $\SL_2(\QQ)$ has no fixed point in $K$. This follows from the case of its subgroup $\SL_2(\ZZ[1/p])$ established in the proof of \Cref{thm:basic:non-amen}.
\end{proof}

\subsection{Organising the layers of non-amenability}
\label{sec:spa}
This last section is purely descriptive and is placed in the context of completely general topological groups $G$ and arbitrary subgroups of $G$.

We propose to consider some sort of ``spectrum'' $\spa G$ recording the layers of non-amenability to be found between subgroups $G$; in particular, $\spa G$ will be reduced to a point if and only if $G$ itself is amenable.

Recall that given any two subgroups $L, H<G$, we say that $L$ is co-amenable to $H$ relative to $G$ if any (jointly continuous) convex compact $G$-space with an $L$-fixed point has an $H$-fixed point.

We write $L\succeq_G H$, of simply $L\succeq H$ when the ambient group does not vary. This is a pre-order relation.

We denote by $\spa G$ the quotient of the set of all subgroups of $G$ under the equivalence relation associated to the pre-order $\succeq_G$. Thus $\spa G$ is a poset and we still denote its order relation by $\succeq_G$ or $\succeq$. We denote the equivalence class of a subgroup $H<G$ by $[H]_G\in \spa G$, or simply $[H]$. 

It is sufficient to consider closed subgroups of $G$ because the closure $\ol H$ of $H$ in $G$ satisfies $[\ol H]=[H]$. Furthermore, $[H]$ only depends on the conjugacy class of $H$. For two subgroups $H,L<G$ we trivially have
\[
H<L \ \Longrightarrow \  [H] \preceq_G [L].
\]
In complete generality, $\spa G$ admits an upper bound, $[G]$, and a lower bound, $[1]$. The former coincides with the set of all co-amenable subgroups of $G$, while the latter coincides with the set of all relatively amenable subgroups of $G$. (Relative amenability, defined in \cite{Caprace-Monod_rel}, boils down to amenability in the case of discrete groups.)

In particular, $\spa G$ is reduced to a point if and only if $[G] = [1]$, which happens if and only if $G$ is amenable. It can happen that $\spa G$ consists precisely of two points, e.g.\ when $G$ is a non-amenable Tarski monster.

\medskip

Regarding functoriality, we note that any morphism $\varphi\colon G\to H$ of topological groups induces a morphism of posets $\spa G \to \spa H$, where for $L<G$ the point $[L]_G$ is mapped to $[\varphi(L)]_H$. This follows immediately from pulling back convex compact $H$-spaces to $G$-spaces via $\varphi$.

This map $\spa G \to \spa H$ is onto if $G\to H$ is onto, but monomorphisms might induce non-injections, for instance in the setting of \Cref{exam:MPP}. Indeed, if $L<G<H$ is such that $L$ is co-amenable in $H$ but not in $G$, then the inclusion morphism $G\to H$ maps the point $[L]_G \in\spa G$ to $[L]_H \in\spa H$, but we have $[L]_G\neq [G]_G$ and $[L]_H = [H]_H = [G]_H$.

It would be interesting to find examples where $\spa G$ can be completely described (without being trivial). We expect that for most familiar discrete groups this object is too large to be described, even though the case of Tarski monsters shows that there are exceptions. 

\medskip

\Cref{thm:more:gaps} can be reformulated as exhibiting a huge part of $\spa G$ for various piecewise-projective groups $G$, as follows. The poset of all sets of prime numbers is isomorphic to the poset of subgroup $\Gamma_S< H(\QQ)$. Then \Cref{thm:more:gaps} states that this uncountable poset is fully faithfully represented (as a poset) into $\spa H(\QQ)$.

More generally, fixing $S$, we can consider the poset of all subsets $T\se S$, which again gives us a poset of subgroups which is uncountable as soon as $S$ is infinite. The non-co-amenability of various $\Gamma_T$ relative to $H(\QQ)$ a fortiori implies non-co-amenability relative to $\Gamma_S$. Thus \Cref{thm:more:gaps} implies:

\begin{cor}\label{cor:spa}
The canonical map $T \to [\Gamma_T]$ is a fully faithful embedding of the poset of all subsets $T \se S$ into $\spa\Gamma_S$.\qed
\end{cor}

On the other hand we expect that non-discrete groups will provide more tractable examples. For instance, what is $\spa G$ for $G=\SL_n(\RR)$\,?


\bibliographystyle{amsalpha}
\bibliography{../../BIB/ma_bib}

\end{document}